\documentclass[11pt]{article}
\PassOptionsToPackage{obeyspaces}{url}
\usepackage{amsfonts, amsmath, amssymb}

\usepackage[hidelinks]{hyperref}
\usepackage{bookmark}
\bookmarksetup{open,numbered,depth=5} 

\usepackage{amsthm}

\usepackage{breakurl}
\usepackage{tikz}

\usepackage{etoolbox} 
\patchcmd{\thebibliography}{\leftmargin\labelwidth}{\leftmargin\labelwidth\addtolength\itemsep{-0.1\baselineskip}}{}{}

\usepackage{mathtools}
\usepackage{enumitem}

\newcommand{\ga}[0]{\alpha }

\newcommand{\gD}[0]{\Delta }

\newcommand{\eps}[0]{\varepsilon }

\newcommand{\vp}[0]{\varphi}

\newcommand{\cE}{\mathcal{E} }

\newcommand{\cG}{\mathcal{G} }

\newcommand{\beq}[1]{\begin{equation}\label{#1}}
\newcommand{\enq}[0]{\end{equation}}

\author{Boris Bukh\thanks{Supported in part through NSF grants DMS-2154063, DMS-2452120 and by Simons Foundation Fellowship.}\,$\,^{\text{,}}$\footnotemark[2] \and Quentin Dubroff\thanks{Department of Mathematical Sciences, Carnegie Mellon University, Pittsburgh, PA 15213, USA\@.}}

\title{Faster random walks via infrequent steering}

\usepackage[nameinlink]{cleveref}
\usepackage{thmtools}

\newtheorem{theorem}{Theorem}
\newtheorem{lemma}[theorem]{Lemma}

\newtheorem{definition}[theorem]{Definition}

\newtheorem{proposition}[theorem]{Proposition}
\newtheorem{remark}[theorem]{Remark}
\newtheorem{question}[theorem]{Question}
\newtheorem{observe}[theorem]{Observation}

\newcommand*{\eqdef}{\stackrel{\mbox{\normalfont\tiny def}}{=}}  
\newcommand*{\veps}{\varepsilon}                                 
\DeclarePairedDelimiter\abs{\lvert}{\rvert}                      
\newcommand*{\cover}{\mathcal{C}}                                

\newcommand*{\B}{\mathcal{B}}

\newcommand*{\E}{\mathbb{E}}
\newcommand*{\Fp}{\mathbb{F}_p}
\newcommand*{\ww}{\mbox{{\sf w}}}
\newcommand*{\roo}{r}
\newcommand*{\pr}{\mathbb{P}}
\newcommand*{\dist}{\text{dist}}

\begin{document}
\maketitle

\begin{abstract}
Random walks on graphs can be slow. To speed them up, imagine that at each step instead of choosing the neighbor at random, there is a small
probability $\varepsilon>0$ that we can choose it. We show that in this case, at least for graphs of bounded degree, there is a way to steer the walk so that it
visits every vertex in $n^{1+o(1)}$ steps with high probability. The key to this result is a way to decompose arbitrary graphs into small-diameter pieces.
\end{abstract}

\paragraph{Ways to speed up random walks.}
Random walks are a popular tool for exploring complicated sets. Many algorithmic advances hinge on the design of random
walks that are quick to traverse their state spaces. That motivated many modifications to standard random walks to speed them up. These include
non-backtracking walks \cite{nonbacktracking}, choosing more favorable initial
states \cite{uniformstart}, running many walks in parallel \cite{manywalks},
and lifting the chain to a larger state space
\cite{lifting}. Here we study another approach:
infrequent adaptive interventions by a controller.

\paragraph{\texorpdfstring{$\veps$}{ε}-biased random walks.}
In this paper, we consider the $\veps$-time biased random walk ($\eps$-TBRW). This model is identical to the simple
random walk on a graph, except that at each step, there is a small probability $\veps>0$ that
a controller gets to choose the next position of the walk (still among the neighbors of the current vertex).
This model was first introduced by Azar, Broder, Karlin, Linial, and Phillips \cite{ABKLP} in the Markovian setting, i.e., where the controller chooses the next step independently of the history of the walk. The authors of \cite{ABKLP} were interested in this model in part as a model of random walks under corrupted randomness, related to models of influence in Boolean functions studied in Ben-Or and Linial's influential paper \cite{BL}. The time-dependent version, where the controller's bias may depend on the history of the walk, was first studied by Georgakopoulos, Haslegrave, Sauerwald, and Sylvester \cite{GHSS,CRW,HSS}. It is closely related to the choice random walk model (CRW), in which the controller chooses each step of the walk among two independent randomly chosen neighbors of the current vertex. In particular, \cite[Proposition 1]{GHSS} shows that $\eps$-TBRW on a graph $G$ of maximum degree $\Delta$ can be simulated by CRW if $\eps \leq  1/ \gD$.

Formally, a \emph{bias function} on a (connected) graph $G$ is a rule which, for each finite walk $(v_0,v_1,\dots,v_t)$ in $G$, specifies a random neighbor of $v_t$. Equivalently, it assigns a probability distribution on $N(v_t)$ to each such sequence.
Given $\eps \in [0,1]$, a bias function $\varphi$, and a starting vertex $x \in V(G)$, the $\veps$-biased random walk on $G$ under $\varphi$ started at $x$ is the sequence $(X_0,X_1, X_2,\dots)$ where $X_0 = x$, and for each $i$, $X_{i+1}$ is a uniformly chosen neighbor of $X_i$ with probability $1-\veps$ and otherwise is chosen according to the distribution $\varphi(X_{[i]})$, where $X_{[i]} = (X_0,\dots, X_i)$. The \emph{cover time} $\tau_{\varphi,x}(\eps)$ is the first time that each vertex of $G$ has been visited by this walk.
We are primarily interested in the optimal cover time
\[T(G,\eps) \eqdef \inf_{\varphi} \max_{x \in V(G)} \E[\tau_{\varphi,x}(\eps)].\]
We could have insisted on the controller choosing the neighbor of $v_t$ deterministically, rather than randomly.
This would not change the value of $T(G,\eps)$: exposing the controller's random choices in advance yields a distribution on
deterministic bias functions, and averaging over this distribution gives the same law of the walk.

\paragraph{New results.} The cover time of the usual random walk can be as large as $4n^3/27+o(n^3)$ \cite{feige_cover}.
In contrast, with a right strategy, $\eps$-biased walks are much faster.
\begin{proposition}\label{thm:nonlin}
    Let $\eps>0$.
    \begin{enumerate}[label=\alph*), ref=Proposition \thetheorem(\alph*)]
    \item For every $n$-vertex graph $G$ we have
      \[
        T(G,\eps)< n^{2+o(1)}.
      \]
    \item
    There is an $n$-vertex graph $G$ with 
    \[T(G, \eps) > \tfrac{1}{256}n^{2-100\eps}.\]
    \end{enumerate}
\end{proposition}
We suspect that part (b) is essentially sharp; specifically, there should be a constant $c > 0$ such that every $n$-vertex graph satisfies $T(G, \eps) < n^{2 - c\eps}$.

Because the construction establishing the lower bound in \Cref{thm:nonlin}(b) has many edges, it is natural to consider sparser graphs.
Let $\cG_\gD(n)$ be the set of connected $n$-vertex graphs with maximum degree at most $\gD$, and let
\[f_\gD(n, \eps) = \max_{G\in \cG_\gD(n)} T(G,\eps).\]
An open question from \cite{OSS} asks:
\begin{question}\label{q:mot}
    Is it true that, for every fixed $\eps>0$ and $\Delta$, there is $C$ such that 
    \[f_\gD(n, \eps) < Cn?\]
\end{question}
\noindent Our principal result is the following bound, which in particular implies \Cref{thm:nonlin}(a).
\begin{theorem}\label{thm:werr}
For any $\gD$ and $\eps > 0$,
\[f_\gD(n,\eps) < \gD \eps^{-1} n^{1 + o(1)}.\]
\end{theorem}
By Markov's inequality, \Cref{thm:werr} also implies the corresponding high-probability statement: multiplying the expected cover time bound by any slowly growing function 
makes the failure probability \(o(1)\), and so the cover time is still \(\Delta \eps^{-1}n^{1+o(1)}\) with high probability.

A positive answer to \Cref{q:mot} would imply that for any pair of vertices there is a strategy so that the commute time between the vertices is $O_{\Delta,\eps}(n)$.
We show that such a strategy does indeed exist. The proof is essentially the same as the proof of \cite[Theorem 1.1]{CRW}.
\begin{proposition}\label{lem:hit}
For any graph $G$ and $u,v \in V(G)$,
    \[\inf_\vp H_\vp(u,v) \leq 2\eps^{-1}e(G)\]
    where $H_\vp(u,v)$ is the expected hitting time of $v$ in an $\eps$-TBRW (with bias function $\vp$) started at $u$ (and $e(G)$ is the number of edges of $G$).
\end{proposition}

\Cref{lem:hit} is proved by analyzing the naive bias strategy $\vp$ that always steers the walk towards~$v$. An interesting question is whether the naive strategy of biasing toward the closest uncovered vertices has expected cover time $n^{1+o(1)}$, or perhaps even $O(n)$. Analyzing this highly non-Markovian strategy seems difficult. For context, we recall an intriguing open problem from \cite{AKchoice} for the CRW model: at each step of CRW on the $\sqrt{n} \times \sqrt{n}$ grid, the controller chooses the vertex (among the two sampled vertices) which has been visited the fewest number of times. Based on numerical evidence, the authors of \cite{AKchoice} conjectured that under this bias strategy, the cover time should be $o(n \log^2 n)$, rather than $\Theta(n\log^2 n)$, which is the cover time of the simple random walk on the grid.

On the $n$-vertex grid $G$, it is not difficult to show that there is a bias strategy for $\eps$-TBRW (and also CRW) which can move from a vertex to a given neighbor in expected time $O(1)$. Thus $\eps$-TBRW can traverse vertices along a spanning tree in $G$ in expected time $O(n)$, which establishes \Cref{q:mot} for the $n$-vertex grid. This strategy can be generalized to any graph of ``uniform polynomial growth,'' meaning that there is a polynomial $p$ such that for each $v \in V(G)$ the set $B^k(v)$, the set of vertices of distance at most $k$ from $v$, satisfies $|B^k(v)| < p(k)$.
\begin{proposition}\label{prop:poly}
    For any polynomial $p$ and $\eps > 0$, there is $C$ such that if $G$ is a graph with $|B^k(v)| < p(k)$ for all $v \in V(G)$ and $k \in \mathbb{N}$, then
    \[T(G,\eps) < Cn.\]
\end{proposition} 

When the polynomial growth condition fails it might take a long time to move from a vertex to a neighbor. For example, on an expander, this takes at least $n^{1 - \Omega(\eps)}$ steps in expectation for any bias strategy. The fact that the bias in $\eps$-TBRW is not strong enough to enforce that the walk stays confined to a specified region is one of the main issues for answering \Cref{q:mot}. 

\paragraph{Proof idea and the main lemma.}
In the context of \Cref{thm:werr}, a particularly easy case is when the graph is an expander: the simple random walk covers an expander in time $O(n \log n)$. More generally, for any graph with hitting times bounded by $h$, the classical Matthews bound \cite{Matth} implies an upper bound of $O(h \log n)$ on the cover time of the simple random walk. In particular, on a graph of diameter $d$ and maximum degree $\gD$, the Matthews bound is $O(d \gD n \log n)$, and thus \Cref{thm:werr} is easy for graphs of diameter $n^{o(1)}$.

In order to prove \Cref{thm:werr}, we cover a general graph by small-diameter pieces and use the Matthews bound on each piece. There are two related issues with this approach: (i) as noted in the previous paragraph, the bias in $\eps$-TBRW is not necessarily powerful enough to keep the walk confined in a given portion of the graph. Without additional information, we can only guarantee that the walk likely stays confined to a ball of radius $O(\log n)$ around any given set. (ii) If a vertex is within distance $O(\log n)$ from many sets in the cover, then it may be visited many times as our strategy attempts to cover the graph. Thus we require a cover of our graph by small-diameter pieces which are in some sense well-separated and not too highly overlapping. We capture these notions formally in the next definition, where \(N(U)\) and \(N^+(U)\eqdef U\cup N(U)\) denote the open and closed neighborhoods of a vertex set \(U\) in \(G\).

\begin{definition}\label{def:rKcover}
    An $(r,K)$-cover of a graph $G$ is a collection of subsets $V_1,\dots, V_m$ of $V(G)$ such that 
    \begin{itemize}
        \item $V(G) = \cup_i V_i,$
        \item the radius of each $G[V_i]$ is at most $r$, and
        \item $\sum_i |N^+(V_i)| \leq Kn.$
    \end{itemize}
\end{definition}

The strategy outlined above \Cref{def:rKcover} is used to establish the following result.
\begin{theorem}\label{thm:bddcov}
    If every $n$-vertex graph admits an $(r,K)$-cover, then 
    \[f_\gD(n,\eps) < 32\eps^{-1} \gD (r+1)Kn  \log^2 n.\]
\end{theorem}
Here and thereafter $\log$ denotes logarithm to base~$2$.

This theorem reduces the probabilistic problem of constructing good strategies for $\eps$-TBRW to an extremal problem about general graphs. \Cref{thm:werr} follows immediately from our next result on the existence of covers.
\begin{theorem}\label{thm:efficient cover}\
\begin{enumerate}[label=\alph*)]
  \item Every graph admits a $(2^{k-1}-1,n^{1/k+o(1)})$-cover for any positive integer $k$.
  \item Every graph admits a $(4^{\sqrt{\log n}}, 4^{\sqrt{\log n}})$-cover if $n$ is large enough.
\end{enumerate}
\end{theorem}

We do not know if the function $4^{\sqrt{\log n}}$ can be replaced by a much slower growing function. In particular, it would be interesting to know if every graph admits a $(\text{polylog } n, \text{polylog } n)$-cover. However,
the bound in part (a) is sharp for small radii. To state the result,
let $\cover_r(n)$ be the smallest number $\cover$ such that every $n$-vertex graph admits an $(r,\cover)$-cover.
\begin{theorem}\label{thm:inefficient cover}
We have
\begin{align*}
  \cover_0(n)&=n,\\
  \cover_1(n)&=n^{1/2+o(1)},\\
  \cover_2(n)&=n^{1/2+o(1)},\\
  \cover_3(n)&=n^{1/3+o(1)}.
\end{align*}
\end{theorem}

\paragraph{Acknowledgments.} This paper was conceived while we were at SLMath in Spring 2025, and we are grateful to NSF for supporting SLMath through grant DMS-1928930.
We thank Alan Frieze and Wes Pegden for several interesting discussions. The first author is grateful to Zilin Jiang for informing him how to spell ``acknowledgments''. 

\section{The maximum hitting time}
In this section, we prove part (b) of \Cref{thm:nonlin} as well as \Cref{lem:hit}. As already mentioned, part (a) of \Cref{thm:nonlin} will follow from \Cref{thm:werr}.

\begin{proof}[Proof of \Cref{thm:nonlin}(b)]
  We construct a graph with the following stronger property: for any vertex $u\in V(G)$, there is $v \in V(G)$ such that for any bias function the expected hitting time from $u$ to $v$ is at least $n^{2 - 100\eps}/256$.

  We first give a construction in the case $n = 2^k - 2$ for $k \in \mathbb{N}$, deferring the simple adjustment for arbitrary $n$ at the end of the proof. The vertex set of $G$ is $V(G) = V_1 \sqcup V_2\sqcup \dots \sqcup V_{k-1}$ where $|V_i| = 2^i$ for each $i \in [k-1]$. The edge set $E(G)$ consists of all possible edges between $V_i$ and $V_{i+1}$ for $i \in [k-2]$.

    Let $u \in V(G)$ be arbitrary, and choose $v \in V_1$ with $v \neq u$.
    Let $R$ be the expected number of visits to $V_{k-1}$ before hitting $V_1$ in a walk started at $V_2$. Since any walk from $u$ to $v$ passes through $V_2$, it suffices to give a lower bound on $R$.
    One optimal bias strategy is the one that always biases the walk toward $V_1$. Indeed, for the quantity $R$, only the sequence of layers visited by the walk matters. Thus any bias strategy induces an $\eps$-biased walk on the path graph with vertex set~$[k-1]$. The problem therefore reduces to the one-dimensional setting, where it is clear that $R$ is minimized by always biasing toward $V_1$.

    Under this strategy, let $Y_t$ be the index $i$ for which $X_t \in V_i$ so that $Y_0,Y_1,Y_2\dots$ is a biased random walk on $[k-1]$. The process $(Y_t)_{t \geq 0}$ may be modeled by a random walk on an edge-weighted $k-1$ vertex path whose $i$th edge has weight $\ga^{i-1}$, where 
    \[\ga\eqdef \frac{(1 - \eps)(4/5)}{\eps + (1 - \eps)/5} = 4 - 20\eps/(4\eps +1)\geq 4(1-5\eps).\]
    Lemma 2.6 from \cite{AF} gives $\pi_{k-1}/\pi_1$ for the expected number of visits to $k-1$ before hitting $1$ in the random walk $(Y_t)$, where $\pi_i$ is the stationary measure of $i$. 
    Using the fact that $2^k >n$ and $\log \ga > 2 - 100\eps$ (here assuming $\eps < 1/50$, since otherwise the result is trivial), we find
    \[R = \pi_{k-1}/\pi_1 = \ga^{k-3}>\tfrac{1}{64}\alpha^k = \frac{1}{64}2^{k\log \ga} > \tfrac{1}{64}n^{2 - 100\eps}.\]

    For arbitrary $n$, choose $k$ with $2^k -2 \leq n < 2^{k+1} -2$, and enlarge $V_{k-1}$ so that the total number of vertices is $n$. In the induced one-dimensional walk on the layers, this only increases the probability of moving from $k-2$ to $k-1$, and hence can only increase $R$. Therefore the same lower bound on $R$ holds:
    \[R \geq \ga^{k-3}>\tfrac{1}{256}\alpha^{k+1} = \frac{1}{256}2^{(k+1)\log \ga} > \tfrac{1}{256}n^{2 - 100\eps}.\qedhere\]
\end{proof}

We now turn to the proof of \Cref{lem:hit}. We in fact prove something a bit stronger. For any $U \subseteq V(G)$ (with $U \neq \emptyset$), consider edge weights $\ww_U$ defined by
\beq{Uwt}\ww_U(e) = (1 - \eps)^{\dist(e, U)}\enq
(where $\dist(e,U)$ is the length of the shortest path incident to both $e$ and $U$). The random walk on $(G,\ww_U)$ is the reversible Markov chain whose transition probabilities from each vertex are proportional to the weights of the incident edges. A useful observation is that there is a bias function, denoted $\vp_U$, whose $\eps$-TBRW emulates the random walk on $(G,\ww_U)$.
\begin{proposition}\label{prop:emu}
    Suppose that $U$ is a non-empty subset of $V(G)$ and $\ww_U(\cdot)$ is defined as in \eqref{Uwt}. There is a bias function $\vp_U$ such that the $\eps$-TBRW under $\vp_U$ has the same law as the random walk on $(G,\ww_U)$.
\end{proposition}
\begin{remark} A nearly identical statement appears as \cite[Prop.\ 4.1]{OSS}, but we repeat the proof here for completeness. Throughout this paper, we will use the notions of the bias function $\vp_U$ and the random walk on $(G,\ww_U)$ interchangeably. \end{remark}
\begin{proof}
    Suppose that $S = (v_0,\dots, v_k)$ is a sequence of vertices satisfying $\dist(v_k,U) = \ell$. Define $N_1(v_k) \eqdef \{w \in N(v_k): \dist(w,U) = \ell-1\}$ and $N_2(v_k) \eqdef N(v_k) \setminus N_1(v_k)$. With 
    \[p\eqdef \frac{(1 - \eps)|N_2|^2}{(|N_1|+|N_2|)(|N_1| + (1 - \eps)|N_2|)},\]
    define $\vp_U(S)$ to be uniform over $N_2$ with probability $p$ 
    and uniform over $N_1$ with probability $1 - p$. Under this Markovian bias function, if $X_t = v_k$, then
    \[\Pr(X_{t+1} = v) =\begin{cases}
        \frac{1- \eps}{|N_1| + |N_2|} + \frac{(1-p)\eps}{|N_1|} &\text{ if } v\in N_1\\
        \frac{1- \eps}{|N_1| + |N_2|} + \frac{p\eps}{|N_2|}  &\text{ if } v\in N_2
    \end{cases}, \]
    which matches the law of the walk on $(G, \ww_U)$. 
\end{proof}
\begin{lemma}\label{lem:hitt}
    For $v \in V(G)$ and $U \subseteq V(G)$, let $H_{\ww_U}(v,U)$ be the expected hitting time of $U$ in a walk started at $v$ in the weighted graph $(G,\ww_U)$ (where $\ww_U$ is defined in \Cref{Uwt}). Then
    \[H_{\ww_U}(v,U) \leq 2\eps^{-1}e(G).\]
    Moreover, if $S = \{w : \dist(w, U) > \dist(v,U) + 2\log_{1/(1 - \eps)} n\}$ and $W = V(G) \setminus S$, then 
    \[H_{\ww_U}(v,U) \leq 2\eps^{-1}(e(G[W]) + 1).\]
\end{lemma}
\noindent An analogous statement (and proof) for the choice random walk appears as \cite[Theorem 3]{GHSS}.
\begin{proof}
    Let $U_i \eqdef \{v: \dist(v,U) = i\}$ and let $G_i$ be the graph induced by $\bigcup_{j \geq i-1} U_j$ after deleting edges with both endpoints in $U_{i-1}$. It suffices to show that $T_i$, the maximum expected hitting time of $U_{i-1}$ over $v \in U_{i}$, satisfies
    \beq{eq:obyo}T_i \leq 2\sum_{e \in G_i} (1 - \eps)^{\dist(e, U_{i-1})}.\enq
    Indeed, given \eqref{eq:obyo}, the first conclusion of the lemma follows by noting that
    \[H_{\ww_U}(v,U) \leq \sum_{i \geq 1} T_i \leq \sum_i 2\sum_{e \in G_i} (1 - \eps)^{\dist(e, U_{i-1})} \leq 2\sum_{e \in G} \sum_{k \geq 0} (1 - \eps)^k = 2\eps^{-1}e(G),\]
    and, similarly, the second conclusion of the lemma follows from
    \begin{align*}
    H_{\ww_U}(v,U) \leq \sum_{i = 1}^{\dist(v,U)} T_i &\leq \sum_{i=1}^{\dist(v,U)} 2\sum_{e \in G_i} (1 - \eps)^{\dist(e, U_{i-1})} \\
    &\leq2 \sum_{e \in G[W]} \sum_{k \geq 0} (1- \eps)^k + 2\sum_{e \not\in G[W]} \sum_{k \geq 2\log_{1/(1 - \eps)}n} (1 - \eps)^k\\
    &\leq 2\eps^{-1} e(G[W]) + 2e(G) \eps^{-1} n^{-2} \leq 2\eps^{-1}(e(G[W]) + 1).
    \end{align*}

    It remains to establish \eqref{eq:obyo}. Suppose $v\in U_i$ and let $u \in U_{i-1}$ be a neighbor of $v$. Let $H_{G_i}(v, u)$ be the expected hitting time of $u$ in a walk started at $v$ and $H^+_{G_i}(u,u)$ be the expected first return time to $u$ in a walk started at $u$. Letting $d_{G_i}(u)$ be the degree of $u$ in $G_i$, we have 
    \[H^+_{G_i}(u,u) - 1 = d_{G_i}(u)^{-1}\sum_{w \in N(u) \cap V(G_i)} H_{G_i}(w,u) \geq d_{G_i}(u)^{-1}H_{G_i}(v,u),\]
    implying
    \beq{eq:hvu} H_{G_i}(v,U_{i-1}) \leq H_{G_i}(v,u) \leq d_{G_i}(u) H^+_{G_i}(u,u).\enq
    To finish, recall \cite[Lemma 2.5]{AF}, which states that $H^+_{G_i}(u,u)$ is $\pi_{C_i}(u)^{-1}$ (the inverse of the stationary measure of $u$ in $C_i$, the connected component of $G_i$ containing $u$). Since all edges incident to $u$ have weight $(1 - \eps)^{i-1}$,
    \[ \pi_{C_i}(u)^{-1} = \frac{2\sum_{e \in C_i} \ww(e)}{d_{G_i}(u)(1 - \eps)^{i-1}} = 2d_{G_i}(u)^{-1}\sum_{e \in C_i} \frac{\ww(e)}{(1 - \eps)^{i-1}} = 2d_{G_i}(u)^{-1}\sum_{e \in C_i}(1 - \eps)^{\dist(e,U_{i-1})}.\]
    Using this inequality in \eqref{eq:hvu} (noting that $C_i \subseteq G_i$) and recalling $v$ was arbitrary establishes \eqref{eq:obyo}.
\end{proof}

\section{The cover time for graphs with polynomial growth}

We prove \Cref{prop:poly} via the following slightly stronger statement, which may be useful elsewhere.
\begin{theorem}
    For all $L$ and $\eps>0$, there is $C$ such that if $G$ is a connected graph with
    \beq{ballbound}\sum_{k=0}^\infty |B^k(v)|(1 - \eps)^{k} < L\enq
    for every $v \in V(G)$,
    then
    \[T(G,\eps) < Cn.\]
\end{theorem}
\begin{proof}   
    Let $x_0 \in V(G)$ be an arbitrary starting vertex. Fix a walk $P = (x_0,\ldots, x_m)$ that covers $V(G)$. We may assume that $m \leq 2n$. Set $C = 4L^2(1 - \eps)^{-1}$.
    
    It suffices to prove that for each $i$, there is a bias strategy $\vp_i$ such that
    \beq{biashit}H_{\vp_i}(x_i,x_{i+1}) < C/2,\enq
    where $H_{\vp_i}(x_i,x_{i+1})$ is the expected hitting time of $x_{i+1}$ in a walk started at $x_i$ under bias strategy $\vp_i$. The result will then follow
    from $T(G,\eps)\leq \sum_i H_{\vp_i}(x_i,x_{i+1})$.

    Set $\vp_i = \vp_U$ as defined in \Cref{prop:emu} with $U = \{x_{i+1}\}$. Equation \eqref{eq:obyo} gives 
    \[H_{\vp_i}(x_i,x_{i+1}) \leq 2 \sum_{e\in G} (1 - \eps)^{\dist(e,x_{i+1})}.\]
    Since \eqref{ballbound} implies that the maximum degree of $G$ is at most $L(1- \eps)^{-1}$, the number of edges at distance at most $k$ from $x_{i+1}$ is at most $|B^k(x_{i+1})|L(1- \eps)^{-1}$. Thus
    \[H_{\vp_i}(x_i,x_{i+1}) \leq 2 \sum_{e\in G} (1 - \eps)^{\dist(e,x_{i+1})} \leq 2\sum_{k \geq 0} |B^k(x_{i+1})|L(1- \eps)^{-1} (1 - \eps)^k < 2(1- \eps)^{-1}L^2,\]
    which gives \eqref{biashit} by our choice of $C$.
\end{proof}

\section{Proof of \texorpdfstring{\Cref{thm:bddcov}}{Proof of Theorem 7}}
Let
\beq{eq:kdef}
k \eqdef -4 \log_{1 - \eps} n.
\enq
Throughout this section we consider a fixed graph $G$ and an $(r,K)$-cover $V_1,\dots, V_m$ of $G^k$, the $k$th power of $G$, where $k$ is given by \eqref{eq:kdef}.
For each $i \in [m]$, let
\[
W_i \eqdef \{w \in V(G): \dist_G(w, V_i) \leq k/2\},
\qquad
\ww_i \eqdef \ww_{_{W_i}},
\qquad
G_i \eqdef (G,\ww_i),
\]
where the weights $\ww_{_{W_i}}$ are defined in \eqref{Uwt}. By \Cref{prop:emu}, for each $i \in [m]$ there is a bias function $\vp_i$ whose $\eps$-TBRW on $G$ has the same law as the random walk on $G_i$.
We now combine the local strategies $\vp_1,\dots,\vp_m$ into a single global bias function $\vp$ by allowing the walk to switch between them during the recursive exploration described below.

\paragraph*{The global strategy.} Consider the auxiliary graph $H$ on $[m]$ where $ij$ is an edge of $H$ iff $W_i \cap W_j \neq \emptyset$. Since $G$ is connected and the sets $V_1,\dots, V_m$ cover $V(G)$, the graph $H$ is connected. Fix a spanning tree $T$ of $H$ and choose a root $\roo$ such that $V_\roo$ contains the starting vertex $X_0$. For each edge $ij \in E(T)$ choose a vertex $v_{ij} \in W_i \cap W_j$. For each $i \in [m]$, let $\hat i$ be the parent of $i$ in $T$, let $C_i$ be the set of children of $i$, and set
\[
R_i \eqdef \{v_{ij}: j \in C_i\}.
\]
Also let $a_\roo \eqdef X_0$ and, for $i \neq \roo$, let $a_i \eqdef v_{\hat i i}$.

The walk performs a random depth-first exploration of $T$. For each $i \in [m]$, we recursively define an exploration $\cE_i$ of the descendant subtree of $T$ rooted at $i$. The exploration $\cE_i$ begins when the walk is at $a_i$. While $\cE_i$ is active, the walk uses the local strategy $\vp_i$. If the walk first hits $v_{ij}$ for some child $j \in C_i$ that has not yet been explored, then control passes from $\cE_i$ to $\cE_j$. The exploration $\cE_i$ remains suspended throughout the entire recursive exploration of the subtree rooted at $j$, and resumes only when that exploration has terminated and the walk has returned to $v_{ij}$. Thus the children of $i$ are explored in the order in which the walk first visits the vertices of $R_i$. After all children of $i$ have been explored, we continue with strategy $\vp_i$ until $W_i$ has been covered \emph{by steps during which $\cE_i$ is active} and the walk returns to $a_i$; at that moment $\cE_i$ terminates. Starting from $X_0 = a_\roo$, we run $\cE_\roo$.

This recursion determines the bias function $\vp$ for every history that occurs before the root exploration terminates. If $S = (x_0,\dots,x_t)$ is such a history, let $i$ be the unique index such that $\cE_i$ is active after the history $S$, and set
\[
\vp(S) \eqdef \vp_i(S).
\]
Once $\cE_\roo$ is finished, we extend $\vp$ arbitrarily to all remaining histories, e.g., by setting $\vp(S) \eqdef \vp_\roo(S)$.

\begin{observe} 
For each $i \in [m]$, let $L_i$ be the number of steps during which $\cE_i$ is active. Then $L_i$ has the same law as the first time that a random walk on $G_i$ started at $a_i$ covers $W_i$ and returns to $a_i$. Moreover,
\[
\tau_{\vp} \leq \sum_{i \in [m]} L_i.
\]
\end{observe}
\begin{proof}
Fix $i \in [m]$, and let $Y^i$ be the subsequence of the global walk consisting of those times at which $\cE_i$ is active. Then $|Y^i| = L_i$. Whenever the exploration of a child $j \in C_i$ is launched, it begins when the walk first hits $v_{ij}$ and ends only when the exploration of the subtree rooted at $j$ is complete; at that moment the walk is again at $v_{ij}$. Thus deleting these child excursions leaves a walk with the same law as the random walk on $G_i$ started at $a_i$, so $Y^i$ has that law.

Since $R_i \subseteq W_i$, once $Y^i$ has covered $W_i$, every child of $i$ has been explored. Hence $\cE_i$ terminates exactly when $Y^i$ has covered $W_i$ and returned to $a_i$. This proves the first claim.

For the second claim, the intervals during which the explorations $\cE_i$ are active partition the time before $\cE_\roo$ terminates. Therefore the number of steps before that time is $\sum_i L_i$. By then each $\cE_i$ has terminated, so every set $W_i$ has been covered. Since $V_i \subseteq W_i$ for every $i$ and the sets $V_i$ cover $V(G)$, the whole graph has been visited. Thus $\tau_{\vp} \leq \sum_i L_i$.
\end{proof}
\begin{lemma}
    For any $i \in [m]$ and $x \in W_i$, if $\tau_i(x)$ is the first time that the random walk on $G_i$ started at $x$ covers $W_i$ and returns to $x$, then
    \beq{eq:taui}\E[\tau_i(x)] < 32(r + 1)\gD \eps^{-1} |N_{G^k}^+(V_i)|\log^2 n.\enq
\end{lemma}
\noindent (Here $\gD$ is the maximum degree of $G$.)

Assume the lemma for now.
By the linearity of expectation,
\[
\E[\tau_{\vp}] \leq \sum_{i \in [m]} \E[L_i].
\]
Since $V_1,\dots,V_m$ is an $(r,K)$-cover of $G^k$,
\[
\sum_{i \in [m]} |N_{G^k}^+(V_i)| \leq Kn.
\]
By the observation, $L_i$ has the same law as $\tau_i(a_i)$ for each $i \in [m]$. Therefore, applying \eqref{eq:taui} with $x=a_i$ for each $i \in [m]$ gives
\[
\E[\tau_{\vp}] < 32(r + 1)\gD \eps^{-1}\log^2 n \sum_{i \in [m]} |N_{G^k}^+(V_i)|
\leq 32(r + 1)\gD \eps^{-1}Kn\log^2 n,
\]
which proves \Cref{thm:bddcov}. 
\begin{proof}[Proof of the lemma]
    Fix $x \in W_i$ and write $\tau_i$ for the first time that the walk started at $x$ covers $W_i$.
    We apply the Matthews bound in the form 
    \[\E\tau_i \leq \max_{u,v \in W_i} H(u,v)\left(1 + 1/2 + \dots + 1/|W_i|\right)\]
    where $H(u,v) = H_{G_i}(u,v)$ is the expected hitting time from $u$ to $v$ in $G_i$
    (see \cite[Equation (11.21)]{LPW}).
    Averaging over $X_{\tau_i}$ and then using the Matthews bound gives  
    \[\E\tau_i(x) \leq \E\tau_i + \max_{u \in W_i} H(u,x)\leq 2\log n\max_{u,v \in W_i} H(u,v).\] 
    
    It suffices to show    
    \beq{kap}
        \max_{u,v \in W_i} H(u,v) < 16\eps^{-1}\gD|N_{G^k}^+(V_i)|(r + 1)\log n.
        \enq
        To that end, we use \cite[Corollary 3.11]{AF}, which states that for any two vertices $u,v\in V(G_i)$
    \beq{commute}H(u,v) + H(v,u)= R_{\mathrm{eff}}(u,v) 2\sum_{e \in G_i} \ww_i(e),\enq
    where $R_{\mathrm{eff}}(u,v)$ is the effective resistance between $u$ and $v$ in $G_i$ (see e.g.\ \cite{AF} or \cite{DoyleSnell} for the definition of $R_{\mathrm{eff}}$). By Rayleigh's Monotonicity Law, $R_{\mathrm{eff}}(u,v)$ is at most the sum of resistances of edges along any walk from $u$ to $v$. So, noting that all edges in $W_i$ have weight (and thus resistance) one, 
    \beq{Reff}R_{\mathrm{eff}}(u,v) \leq k(2r+1)\enq
    will hold if there is a path connecting $u$ and $v$ of length $k(2r+1)$ lying entirely in $W_i$. To find this path, note that if (respectively) $u'$ and $v'$ are the closest vertices to $u$ and $v$ in $V_i$, then there are paths $P_u$ and $P_v$ of length at most $k/2$ in $W_i$ connecting $u$ to $u'$ and $v$ to $v'$. Additionally, in $G^k[V_i]$, there is a path of length at most $2r$ connecting $u'$ and $v'$, and since each $e \in G^k[V_i]$ corresponds to a path of length at most $k$ in $G[W_i]$, we can find a walk $P'$ of length at most $2rk$ in $G[W_i]$ connecting $u'$ and $v'$. Concatenating the walks $P_u$, $P'$ and $P_v$ results in a walk of length at most $(2r+1)k$ connecting $u$ and $v$ in $G[W_i]$.

    Finally, with $S = V(G) \setminus N_{G^{k}}^+(V_i)$, all edges incident to $S$ have weight at most $(1 - \eps)^{k/2} = n^{-2}$, so 
    \beq{wtsum}2\sum_e \ww_i(e) \leq |S|\gD n^{-2} + \gD |N_{G^{k}}^+(V_i)|< 1+ \gD |N_{G^{k}}^+(V_i)|\leq  2\gD|N_{G^k}^+(V_i)|.\enq
    Using \eqref{Reff} and \eqref{wtsum} in \eqref{commute} gives
    \[\max_{u,v \in W_i} H(u,v) < 2\gD |N_{G^k}^+(V_i)|k(2r+1) < 4\gD |N_{G^k}^+(V_i)|k(r+1) .\]
    The result then follows from the fact (recall \eqref{eq:kdef}) that
    \[k = -4\log_{1 - \eps}n \leq 4\eps^{-1}\ln 2 \log n < 4 \eps^{-1}\log n.\qedhere\]
\end{proof}

\section{Graph covers}

\paragraph{Proof of \texorpdfstring{\Cref{thm:efficient cover}}{Theorem 8} (existence of efficient covers).}
Let $G$ be any $n$-vertex graph. For a vertex $v\in V(G)$ and an integer $r$, we denote by $B(v,r)$ the closed ball of radius $r$ centered at $v$. For a number $0<p\leq 1$, let $V(p)$ be the random subset of $V(G)$ obtained by picking each vertex of $G$ independently with probability~$p$. Define
\[t_i\eqdef (n\log n)^{i/k}, \quad \text{and} \quad 
  p_i\eqdef \min(1,t_i^{-1}\cdot 2\log n)\quad\quad\text{for }i=0,1,\dotsc,k.\]
Consider the $k+1$ independent random sets $V(p_0),\dots, V(p_k)$, and for $i = 0,1,\dots, k-1$, define
\[  S_i\eqdef \{v \in V(G) : \abs{B(v,2^i)}\leq t_{i+1} \} \quad \text{ and }\quad
  \B_i\eqdef \{B(v,2^i-1) : v\in S_i\cap V(p_i)\}.\]
We claim that $\B=\B_0\cup \B_1\cup\dotsb\cup \B_{k-1}$ is an $(2^{k-1}-1,n^{1/k+o(1)})$-cover of $G$ with positive probability.

We claim that 
\[\mbox{for each $v\in V(G)$, there is $i=i(v)\in \{0,\dots, k-1\}$ such that $\abs{B(v,2^i-1)\cap S_i}\geq t_i$.}\] Suppose for the sake of contradiction that the claim fails for some $v\in V(G)$; we show by induction on $i$ that $B(v,2^i-1)$ contains at least $t_i$ vertices:
The base case
$i=0$ is obvious. If the induction hypothesis for $i$ holds, it follows that
$B(v,2^i-1)\setminus S_i$ is non-empty, say $u_i\in B(v,2^i-1)\setminus S_i$. Because $B(v,2^{i+1}-1)\supseteq B(u_i,2^i)$, this means that
$\abs{B(v,2^{i+1}-1)}\geq \abs{B(u_i,2^i)}\geq t_{i+1}$, as required. However, $t_k > n$, so it cannot be that $B(v,2^k-1) \geq t_k$, so we have reached a contradiction.

Using the claim, we find
\begin{align*}
  \pr[v\text{ is not covered by }\B_{i(v)}]&= \pr\bigl[B(v,2^i-1)\cap S_i\cap V(p_i)=\emptyset\bigr]\\
  &=(1-p_i)^{\abs{B(v,2^i-1)\cap S_i}}< 1/n^2,
\end{align*}
and hence $\pr[\B\text{ is not a cover of }V(G)]< 1/n$.

Since
\[
  \E\left[\sum_{B\in \B}\abs{N^+(B)}\right]\leq \sum_{i=0}^{k-1} p_i\sum_{v\in S_i}\abs{B(v,2^i)}\leq \sum_{i=0}^{k-1}p_i t_{i+1}\abs{S_i}\leq k n^{1+1/k}2\log^{1+ 1/k}n,
\]
Markov's inequality implies that the family $\B$ is a $(2^{k-1}-1,4k n^{1/k}\log^{1+ 1/k}n)$-cover of $G$ with positive probability. This completes the proof of
part (a) of \Cref{thm:efficient cover}.\medskip

Part (b) follows by setting $k=2\sqrt{\log n}$ and noting that $4k n^{1/k}\log^{1+ 1/k}n < 2^{2\sqrt{\log n}}$ for reasonably large $n$.

\paragraph{Proof of \texorpdfstring{\Cref{thm:inefficient cover}}{Theorem 9} (optimal covers of small radius).}
Since radius-$0$ sets are singletons, it follows that $\cover_0(n)\leq n$,
which is sharp for $K_n$.

Since $\cover_r(n)$ is monotone nonincreasing as a function of $r$, and the inequalities $\cover_1(n)\leq n^{1/2+o(1)}$,
$\cover_3(n)\leq n^{1/3+o(1)}$ follow from \Cref{thm:efficient cover},
it suffices to prove that $\cover_2(n)\geq \tfrac{1}{32}n^{1/2}$ and $\cover_3(n)\geq \tfrac{1}{400}n^{1/3}$.\medskip

\textsc{Construction showing that $\cover_2(n)\geq \tfrac{1}{32}n^{1/2}$.}
By Bertrand's postulate it suffices to show that $\cover_2(n)\geq n^{1/2}/16$ for $n$ of the form $n=2p(p-1)$, where $p$ is prime.

So, let $p$ be a prime number, and let $\Gamma$
be the set of all functions of the form $x\mapsto \lambda x+u$ with $\lambda\in \Fp^\times$ and $u\in \Fp$. The
set $\Gamma$ forms a group under composition. Let $H$ be the subgroup $\{x\mapsto \lambda x : \lambda\in \Fp^\times\}$,
and let $S\eqdef \{s_0,s_1\}$ where $s_0\eqdef \operatorname{id}$ and $s_1\eqdef (x\mapsto x+1)$.

Our example is the bipartite Cayley graph with the generating
set $SH$, i.e., its vertex set is $X\cup Y$ where $X$ and $Y$ are disjoint
copies of $\Gamma$, with vertex $f\in X$ being connected to all vertices in $Y$ that are
of the form $(sh)f$ for some $s\in S$ and $h\in H$.

To simplify the notation, given a subset $C$ of $\Gamma$ there is a copy of $C$ inside
$X$, and a copy of $C$ inside $Y$; we denote them by $C\cap X$ and $C\cap Y$, respectively.

Suppose that $V_1,\dotsc,V_m$ form a $(2,K)$-cover of this graph $G$. We claim that
\begin{equation}\label{eq:radtwo}
  \abs{N^+(V_i)}\geq \frac{p-1}{4} \abs{V_i\cap Y}\qquad\text{for all }i.
\end{equation}
From this it would follow that $\sum_i \abs{N^+(V_i)}\geq \frac{p-1}{4}\abs{\Gamma}=\frac{p-1}{8}n \geq n^{3/2}/16$.

We now prove \eqref{eq:radtwo}. Let $f_i$ be the center of $V_i$, so that $V_i\subseteq B(f_i,2)$.
Because of symmetry, we may assume that $f_i$ is equal to $\operatorname{id}$ (either in $X$ or in $Y$).

To prove \eqref{eq:radtwo}, we split into cases according to the location of $f_i$:
\begin{itemize}
  \item If $f_i\in X$, then $V_i\cap Y\subseteq N(f_i)=SH\cap Y$.
    Therefore $N^+(V_i)$ contains the set $(SH)^{-1}(V_i\cap Y)=HS^{-1}(V_i\cap Y)$. Note that
    if $y\in Y$, then the neighborhood of $y$ is $HS^{-1}y=Hs_0^{-1}y\cup Hs_1^{-1}y$.
    Because the cosets $Hs_1^{-1}y$ are disjoint for distinct $y\in s_0H$, and
    the sets $Hs_0^{-1}y$ are disjoint for distinct $y\in s_1H$, it follows that
    \begin{align*}
      \abs{N^+(V_i)}&\geq \max\bigl(\abs{Hs_0^{-1} (V_i\cap Y\cap s_1H)},\abs{Hs_1^{-1} (V_i\cap Y\cap s_0H)}\bigr)\\
                  &=\abs{H}\max\bigl(\abs{V_i\cap Y\cap s_1H},\abs{V_i\cap Y\cap s_0H}\bigr)\\
                  &\geq \abs{H}\abs{V_i\cap Y}/2\\
                  &=\frac{p-1}{2}\abs{V_i\cap Y}.
    \end{align*}
  \item If $f_i\in Y$, then $V_i\cap Y\subseteq \{f_i\}\cup N(N(f_i))=\bigl(\{\operatorname{id}\}\cup (SH)(SH)^{-1}\bigr)\cap Y=SHS^{-1}\cap Y$.
    As above, the cosets $Hs_1^{-1}y$ are disjoint for distinct $y\in s_0Hs_0^{-1}$. They are also disjoint for distinct $y\in s_0Hs_1^{-1}$.
    Similarly, the cosets $Hs_0^{-1}y$ are disjoint for distinct $y\in s_1Hs_0^{-1}$, and also for distinct $y\in s_1Hs_1^{-1}$. So,
    \begin{align*}
      \abs{N^+(V_i)}&\geq \max_{j,k\in \{0,1\}}\abs{Hs_j^{-1}(V_i\cap Y\cap s_{1-j}Hs_k^{-1})}\\
                    &=\abs{H}\max_{j,k\in \{0,1\}} \abs{V_i\cap Y\cap s_{1-j}Hs_k^{-1}}\\
                    &\geq \frac{p-1}{4}\abs{V_i\cap Y}.
    \end{align*}    
\end{itemize}

\textsc{Construction showing that $\cover_3(n)\geq \tfrac{1}{200}n^{1/3}$ for $n$ of the form $n=2p^2(p-1)$.}
This is similar to the above, except for the choice of the group and the generating set.
Define
\begin{align*}
  \Gamma'&\eqdef \{(x\mapsto \lambda x+u_1,x\mapsto \lambda x+u_2) : \lambda\in \Fp^\times,\ u_1,u_2\in \Fp\}.\\
\intertext{Note that $\Gamma'$ is a subgroup of $\Gamma\times \Gamma$. For a vector $\mathbf v=(v_1,v_2)\in \{0,1\}^2$, we write $s_{\mathbf v}\eqdef\nobreak (s_{v_1},s_{v_2})\in \Gamma'$, where $s_0$ and $s_1$ are as above. Let $Q\eqdef \{(0,0),(0,1),(1,0)\}$, and let}
  H'&\eqdef \{(x\mapsto \lambda x,x\mapsto \lambda x) : \lambda\in \Fp^\times\},\\
  S'&\eqdef \{s_{\mathbf v}:\mathbf v\in Q\}.
\end{align*}
Let $G'$ be the bipartite Cayley graph on the group $\Gamma'$ with the generating set $S'H'$, i.e.,
the parts $X',Y'$ are copies of $\Gamma'$ and $f\in X'$ is connected to all $(sh)f$ with $s\in S'$ and $h\in H'$.
Let $V_1,\dotsc,V_m$ be a $(3,K)$-cover of $G'$. We claim that
\[
  \abs{N^+(V_i)}\geq 3^{-3}(p-1) \abs{V_i\cap Y'}.
\]
The result then follows, since we get $\sum_i \abs{N^+(V_i)}\geq 3^{-3}(p-1)\sum_i \abs{V_i\cap Y'} \geq 3^{-3}(p-1)n/2 = 3^{-3}n^2/(4p^2) \geq 2^{-2}3^{-3}n^{4/3}\geq \tfrac{1}{200}n^{4/3}$.

As before, let $f_i$ be the center of $V_i$, which we may assume to be equal to $(\operatorname{id},\operatorname{id})$ (in $X'$ or in $Y'$).
\begin{itemize}
\item Suppose that $f_i\in X'$. Then
  \[
    V_i\cap Y'\subseteq N(f_i)\cup N(N(N(f_i)))=S'H'(S')^{-1}S'H' \cap Y'.
  \]
  We claim that, for any $\mathbf b,\mathbf c,\mathbf d\in Q$, there is a choice of $\mathbf a\in Q$ such that
  the sets $H' s_{\mathbf a}^{-1}y$ are disjoint for distinct $y\in s_{\mathbf b}H' s_{\mathbf c}^{-1}s_{\mathbf d}H'$.
  From the claim it would then follow that
  \begin{align*}
    \abs{N^+(V_i\cap Y')}&\geq \abs{H'}\max_{\mathbf b,\mathbf c,\mathbf d\in Q}\abs{V_i\cap Y'\cap s_{\mathbf b}H' s_{\mathbf c}^{-1}s_{\mathbf d}H'}\\&\geq 3^{-3}(p-1)\abs{V_i\cap Y'}.
  \end{align*}

  We now prove the claim. Let $\mathbf b,\mathbf c,\mathbf d\in Q$ be arbitrary. If $\mathbf c=\mathbf d$, take any $\mathbf a\neq\mathbf b$. Otherwise choose $\mathbf a$ so that $\mathbf b-\mathbf a$ and $\mathbf d-\mathbf c$ are linearly independent over $\Fp$; this is possible since no line contains all three vectors in $Q$. Write elements of $\Gamma'$ as pairs $(r,\mathbf u)$, where $r\in\Fp^\times$ is the common slope and $\mathbf u\in\Fp^2$ is the translation vector. Every $y\in s_{\mathbf b}H' s_{\mathbf c}^{-1}s_{\mathbf d}H'$ has the form $(\rho,\mathbf b+\lambda(\mathbf d-\mathbf c))$ for some $\rho,\lambda\in\Fp^\times$. Thus, if the cosets corresponding to $y=(\rho,\mathbf b+\lambda(\mathbf d-\mathbf c))$ and $y'=(\rho',\mathbf b+\lambda'(\mathbf d-\mathbf c))$ meet, then
  comparison of translation vectors gives
  \[
    \mu(\mathbf b-\mathbf a+\lambda(\mathbf d-\mathbf c))=\mu'(\mathbf b-\mathbf a+\lambda'(\mathbf d-\mathbf c))
  \]
  for some $\mu,\mu'\in\Fp^\times$,
  and comparison of slopes gives $\mu\rho=\mu'\rho'$. When $\mathbf c\neq\mathbf d$, linear independence gives $\lambda=\lambda'$ and $\mu=\mu'$, while when $\mathbf c=\mathbf d$, the condition $\mathbf a\neq\mathbf b$ gives $\mu=\mu'$. In either case, $\rho=\rho'$ as well, and hence $y=y'$. 

\item If $f_i\in Y'$, then $V_i\cap Y'\subseteq \{f_i\}\cup N(N(f_i))=S'H'(S')^{-1}$. For fixed $\mathbf b,\mathbf c\in Q$, choose any $\mathbf a\in Q$ with $\mathbf a\neq\mathbf b$. The sets $H' s_{\mathbf a}^{-1}y$ are disjoint for distinct $y\in s_{\mathbf b}H's_{\mathbf c}^{-1}$. Indeed, every such $y$ has the form $(\lambda,\mathbf b-\lambda\mathbf c)$, and if the cosets corresponding to $(\lambda,\mathbf b-\lambda\mathbf c)$ and $(\lambda',\mathbf b-\lambda'\mathbf c)$ meet, then, for some $\mu,\mu'\in\Fp^\times$,
  \[
    \mu(\mathbf b-\mathbf a-\lambda\mathbf c)=\mu'(\mathbf b-\mathbf a-\lambda'\mathbf c)
    \qquad\text{and}\qquad
    \mu\lambda=\mu'\lambda'.
  \]
  Adding $\mu\lambda \mathbf c=\mu'\lambda'\mathbf c$ to both sides of the first equation yields $\mu=\mu'$, in view of $\mathbf a\neq\mathbf b$.
  Hence $\lambda=\lambda'$, and so
  \begin{align*}
    \abs{N^+(V_i\cap Y')}&\geq \abs{H'}\max_{\mathbf b,\mathbf c\in Q}\abs{V_i\cap Y'\cap s_{\mathbf b}H's_{\mathbf c}^{-1}}\\&\geq 3^{-2}(p-1)\abs{V_i\cap Y'}.
  \end{align*}
\end{itemize}

\bibliographystyle{plain}
\bibliography{TBRW}
\end{document}